\documentclass[11pt,a4paper]{article}
\usepackage{amsthm,amsmath,amssymb,amsfonts,hyperref}
\newtheorem{theorem}{Theorem}

\newtheorem{definition}{Definition}

\numberwithin{equation}{section}
\numberwithin{lemma}{section}
\numberwithin{theorem}{section}
\numberwithin{corollary}{section}

\allowdisplaybreaks

\begin{document}
\setcounter{page}{1}
\title{Some extended hypergeometric matrix  functions and their fractional calculus}

\author{Ashish Verma \footnote{Department of Mathematics, Prof. Rajendra Singh (Rajju Bhaiya), Institute of Physical Sciences for Study and Research   V.B.S. Purvanchal University, Jaunpur  (U.P.)- 222003, India. \newline Email: vashish.lu@gmail.com (Corresponding author)}, \, Ravi Dwivedi \footnote{Department of Basic Science and Humanities, PSIT College of Engineering, Kanpur 209305, India.\newline Email: dwivedir999@gmail.com} \, and Vivek Sahai \footnote{Department of Mathematics and Astronomy, Lucknow University, Lucknow 226007, India.\newline Email: sahai\_vivek@hotmail.com}}

\maketitle
\begin{abstract}
In this paper, we study some extended hypergeometric functions from matrix point of view. We have given the integral representations of these matrix functions. Finally, we obtain some generating function relations using fractional derivative operators.

\medskip
\noindent\textbf{Keywords}: Matrix functional calculus, Appell functions, Lauricella functions, Fractional derivative operators

\medskip
\noindent\textbf{AMS Subject Classification}: 15A15; 33C65. 
\end{abstract}

\section{Introduction}
The Gauss hypergeometric function and Kummer hypergeometric function are amongst the most appeared special functions in mathematics as well as in physics. Due to a wide range of applications, these hypergeometric functions have been studied from different aspects \emph{viz.} from  matrix point of view \cite {jjc98a, AM, ajt}; extended Gauss and Kummer hypergeometric functions \cite{cqs} and finite field analogue \cite{jg} to name a few. Recently, the matrix version of extended beta function,  extended Gauss hypergeometric function and Kummer hypergeometric function have been studied by Abdalla and Bakhet in \cite{ab1, ab2}. The regions of convergence, integral representations and differential formulas satisfied by these matrix functions are determined.   Verma   discussed recursion formula, infinite summation formula for the Srivastava's triple hypergeometric matrix functions $H_{\mathcal{A}}$, $H_{\mathcal{B}}$ and $H_{\mathcal{C}}$ \cite{A4}. In \cite{A1, A2, A3,  AP}, authors  introduced the  incomplete first, second and fourth  Appell hypergeometric matrix functions and  incomplete Srivastava's  triple  hypergeometric matrix  functions  and studied some basic properties;  matrix differential equation, integral formula, recursion formula, recurrence relation and differentiation formula of these functions. In this paper, we study the matrix analogues of some extended hypergeometric functions of two and three variables. We obtain the  integral representations for these matrix functions. Finally, we use the extended Riemann-Liouville derivative operator and obtain certain generating function relations in terms of these matrix functions. The section-wise treatment is as follows:    

In Section~2, we give the basic definitions related to special matrix functions that are needed in the sequel.  In Section~3, we define the extended Appell matrix functions and extended Lauricella matrix function.  We also discuss the regions of convergence as well as the integral representations satisfied by these matrix functions. In Section~4, we define the matrix analogue of extended Riemann-Liouville derivative operators and transform several matrix functions under these derivative operators. Finally, in Section~5, we obtain certain generating function relations which turn out in terms of extended Appell matrix functions and extended Lauricella matrix function.

\section{Preliminaries}
Throughout the paper, $\mathbb{C}^{r\times r}$ is the vector space of $r$-square matrices with complex entries. For a matrix   $A\in \mathbb{C}^{r\times r}$ the spectrum, denoted by $\sigma(A)$, is the set of eigenvalues of $A$. 
If $\Re(z)$ denotes the real part of a complex number $z$, then a matrix $A$ in $\mathbb{C}^{r\times r}$  is said to be positive stable if $\Re(\lambda)>0$ for all $\lambda\in\sigma(A)$. 
 
 If $A$ is a positive stable matrix in $\mathbb{C}^{r \times r}$, then $\Gamma(A)$ can be expressed as \cite{jjc98a}
 \begin{equation}
 \Gamma(A) = \int_{0}^{\infty} e^{-t} \, t^{A-I}\, dt.\label{1a1.4}
 \end{equation} 
Furthermore, if $A+kI$ is invertible for all integers $k\geq 0$, then the reciprocal gamma matrix function is defined as \cite{jjc98a}
\begin{equation}
\Gamma^{-1}(A)= A(A+I)\dots (A+(n-1)I)\Gamma^{-1}(A+nI) , \  n\geq 1.\label{eq.07}
\end{equation}
If $M \in \mathbb{C}^{r\times r}$ is a positive stable matrix  and $n\geq 1$ is an integer, then the  gamma matrix function can also be defined in the form of a limit as \cite{jjc98a}
\begin{equation} 
\Gamma (M) = \lim_{n \to \infty} (n-1)! \, (M)_n^{-1} \, n^M. \label{c1eq10} 
\end{equation}
By  application of the matrix functional calculus, the Pochhammer symbol  for  $A\in \mathbb{C}^{r\times r}$ is given by \cite{jjc98b}
\begin{equation}
(A)_n = \begin{cases}
I, & \text{if $n = 0$},\\
A(A+I) \dots (A+(n-1)I), & \text{if $n\geq 1$}.
\end{cases}\label{c1eq.09}
\end{equation}
This gives
\begin{equation}
(A)_n = \Gamma^{-1}(A) \ \Gamma (A+nI), \qquad n\geq 1.\label{c1eq.010}
\end{equation} 
 If $A$ and $B$ are positive stable matrices in $\mathbb{C}^{r \times r}$, then, for $AB = BA$, the beta matrix function is defined as \cite{jjc98a}
\begin{align}
\mathfrak{B}(A,B) =\Gamma(A)\Gamma(B)\Gamma^{-1}(A+B) &= \int_{0}^{1}t^{A-I}(1-t)^{B-I}dt \label{1ca1.4}\\
&=\int_{0}^{\infty}u^{A-I}(1+u)^{-(A+B)}du.\label{1ca1.5}
\end{align}
Let $A$, $B$ and $\mathbb{X}$ be positive stable  and commuting matrices in $\mathbb{C}^{r\times r}$ such that $A+kI$,  $B+kI$ and $\mathbb{X}+kI$ are invertible for all integer $k\geq 0$. Then the extended beta matrix function $\mathfrak{B}(A, B; \mathbb{X})$ is defined by \cite{ab1}
\begin{align}
\mathfrak{B}(A, B; \mathbb{X})=\int_{0}^{1} t^{A-I} (1-t)^{B-I} \exp\Big(\frac{-\mathbb{X}}{t(1-t)}\Big)dt.\label{xb1}
\end{align}
Hence,
\begin{align}
\mathfrak{B}(A, B; \mathbb{X})= \Gamma({A}, \mathbb{X}) \,\Gamma({B}, \mathbb{X})\, \Gamma^{-1}({A+B}, \mathbb{X}).
\end{align}
It is obvious that $\mathbb{X}=O$ gives the original beta matrix  function \cite{jjc98a}.
Let $A$, $B$, $C$ be positive stable matrices in $\mathbb{C}^{r\times r}$ such that $C+kI$ is invertible for all integers $k\ge 0$. Then the Gauss hypergeometric matrix function is defined by \cite{jjc98b}
\begin{align}
{}_2F_1 (A, B; C; z) = \sum_{n=0}^{\infty} (A)_n (B)_n (C)_n^{-1} \frac{z^n}{n!}.\label{52.9}
\end{align}
The series \eqref{52.9} converges absolutely for $\vert z\vert < 1$ and for $z = 1$, if $\alpha(A) + \alpha(B) < \beta(C)$, where $\alpha(A) = \max\{\, \Re(z) \mid z\in \sigma(A)\, \}$, $\beta(A) = \min\{\, \Re(z) \mid z \in \sigma(A)\, \}$ and $\beta(A) = -\alpha(-A)$.

Furthermore, if $CB=BC$ and $C$, $B$ and $C-B$ are positive stable, then for $|z|<1$ an integral representation of \eqref{52.9} is given by  \cite{jjc98b}
\begin{align}
{_2F_{1}}(A, B; C; z) = &\left(\int_{0}^{1} (1-z t)^{-A}\,t^{B-I} (1-t)^{C-B-I}dt\right)\nonumber\\
&\times\Gamma^{-1}(B)\Gamma^{-1}(C-B) \Gamma(C).
\end{align} 
Recently,  Abdalla et{.al} generalized the Gauss and  Kummer hypergeometric matrix function. Let $A$, $B$, $C$, $C-B$ and $\mathbb{X}$ be positive stable matrices in $\mathbb{C}^{r \times r}$ such that $CB = BC$, $C\mathbb{X} = \mathbb{X}C$ and $B \mathbb{X} = \mathbb{X} B$. Then the extended Gauss hypergeometric matrix function (EGHMF) and  extended Kummer hypergeometric matrix function (EKHMF) are defined by \cite{ab2}
\begin{align}
&F^{(\mathbb{X})}(A, B; C; z)\notag\\
&=  \left(\sum_{m\geq 0} (A)_{m} \, \mathfrak{B}(B+mI, C-B; \mathbb{X})\frac{z^{m}}{m!}\right) \times \Gamma(C)\Gamma^{-1}(B)\Gamma^{-1}(C-B);\label{eg1}
\end{align}
and
\begin{align}
&\phi^{(\mathbb{X})}(B; C; z)\notag\\
&=  \left(\sum_{m\geq 0} \, \mathfrak{B}(B+mI, C-B; \mathbb{X})\frac{z^{m}}{m!}\right) \times \Gamma(C)\Gamma^{-1}(B)\Gamma^{-1}(C-B)\label{kh1}
\end{align}
respectively.

Integral representation of extended matrix functions given in Equations \eqref{eg1} and \eqref{kh1} are as follows
\begin{align}
F^{(\mathbb{X})}(A, B; C; z) & = \left(\int_{0}^{1}  (1-zt)^{-A} t^{B-I} (1-t)^{C-B-I} \exp\Big(\frac{-\mathbb{X}}{t(1-t)}\Big)dt\right)\notag\\
& \quad \times \Gamma(C)\Gamma^{-1}(B)\Gamma^{-1}(C-B);\label{ieg1}
\end{align}
\begin{align}
&\phi^{(\mathbb{X})}(B; C; z) \notag\\
& = \left(\int_{0}^{1} t^{B-I} (1-t)^{C-B-I} \exp\Big(\frac{-\mathbb{X}}{t(1-t)}\Big)dt\right) \times \Gamma(C)\Gamma^{-1}(B)\Gamma^{-1}(C-B).\label{ikh1}
\end{align}
\section{The extended  hypergeometric matrix  functions of two and three variables}
In this section, we introduce extended Appell matrix  functions (EAMF's) and extended Lauricella's matrix  function (ELMF) of three variables. More explicitly, we give the extended form of  Appell matrix  functions  $F_{1}(A, B, B'; C; z, w)$, $F_{2}(A, B, B'; C, C'; z, w)$ and  Lauricella matrix  function of three variables $F^{(3)}_{D}(A, B, B', B''; C; z, w, v)$, \cite {ds1,ds4, ds5}, in terms of the extended beta matrix function. We also give here the integral representations for these extended hypergeometric matrix functions. 

Let $A$, $B$, $B'$, $C$, $C-A$ and $\mathbb{X}$ be positive stable matrices in $\mathbb{C}^{r \times r}$ such that $A$, $C$, $\mathbb{X}$ commutes with each other and $CB = BC$, $CB' = B'C$. Then, we define extended Appell  hypergeometric matrix  function $F_{1}(A, B, B'; C;z, w; \mathbb{X})$ as
\begin{align}
&F_{1}(A, B, B'; C;z, w; \mathbb{X})\notag\\
&= \Gamma\left(\begin{array}{c}C\\ A, C-A\end{array}\right) \sum_{m, n\geq 0}\mathfrak{B}(A+(m+n)I, C-A; \mathbb{X}) \, (B)_{m} (B')_{n}\frac{z^{m} w^{n}}{m! n!}, \label{2eq1}
\end{align}
where $\Gamma\left(\begin{array}{c}C\\ A, C-A\end{array}\right) = \Gamma(C) \Gamma^{-1} (A) \Gamma^{-1} (C-A)$. 

For positive stable matrices $A$, $B$, $B'$, $C$, $C'$, $C-B$, $C'-B'$ and $\mathbb{X}$ in $\mathbb{C}^{r \times r}$ such that $B$, $B'$, $C$, $C'$  and $\mathbb{X}$ commutes with each other, we define the extended Appell  hypergeometric matrix  function $F_{2}(A, B, B'; C, C'; z, w; \mathbb{X})$ as
\begin{align}
&F_{2}(A, B, B'; C, C'; z, w; \mathbb{X})\notag\\&= \sum_{m, n\geq 0}(A)_{m+n} \mathfrak{B}(B+mI, C-B; \mathbb{X}) \mathfrak{B}(B'+nI, C'-B'; \mathbb{X})\frac{z^{m}w^{n}}{m! n!} \notag\\
& \quad \times \Gamma\left(\begin{array}{c}C, C'\\ B, B', C-B, C'-B'\end{array}\right).\label{2eq2}
\end{align}
Suppose $A$, $B$, $B'$, $B''$, $C$, $C-A$ and $\mathbb{X}$ be positive stable matrices in $\mathbb{C}^{r \times r}$ such that $A$, $C$, $\mathbb{X}$ commutes with each other and $CB = BC$, $CB' = B'C$, $CB'' = B'' C$. Then, we define the extended Lauricella  hypergeometric matrix  function $F^{(3)}_{D, \mathbb{X}}(A, B, B', B''; C ; z, w, v)$ as
\begin{align}
&F^{(3)}_{D, \mathbb{X}}(A, B, B', B''; C ; z, w, v)\notag\\
& = \Gamma\left(\begin{array}{c}C\\ A, C-A\end{array}\right) \sum_{m, n, p\geq 0}\mathfrak{B}(A+(m+n+p)I, C-A; \mathbb{X}) (B)_{m} (B')_{n} (B'')_{p} \frac{z^{m} w^{n}v^{p}}{m! n! p!}.\label{2eq3}
\end{align}
We now turn our attention in finding the integral representations of extended Appell matrix  functions (EAMF's) and extended Lauricella  matrix  function (ELMF) of three variables. We start with the integral representation of  $F_{1}(A, B, B'; C; z, w; \mathbb{X})$ determined in the next theorem.
\begin{theorem}\label{t1} Let $A$, $B$, $B'$, $C$, $C-A$ and $\mathbb{X}$ be positive stable matrices in $\mathbb{C}^{r \times r}$ such that $A$, $C$, $\mathbb{X}$ commutes with each other and $CB = BC$, $CB' = B'C$. Then  the  EAMF $F_{1}(A, B, B'; C; z, w; \mathbb{X})$ can be presented in the integral form as 
\begin{align}
F_{1}(A, B, B'; C; z, w; \mathbb{X}) &
=   \Gamma\left(\begin{array}{c}C\\ A, C-A\end{array}\right) \int_{0}^{1} u^{A-I} (1-u)^{C-A-I} (1-zu)^{-B} \nonumber\\
 &\quad \times(1-wu)^{-B'}\exp\Big(\frac{-\mathbb{X}}{u(1-u)}\Big)du.
\end{align}
\end{theorem}
\begin{proof}  Using  the  integral representation of extended beta matrix function from \eqref{xb1}  in the definition of  the EAMF $F_{1}(A, B, B'; C; z, w; \mathbb{X})$, we get
\begin{align}
F_{1}(A, B, B'; C; z, w; \mathbb{X}) &= \Gamma\left(\begin{array}{c}C\\ A, C-A\end{array}\right) \sum_{m, n\geq 0}^{} \int_{0}^{1} u^{A-I} (1-u)^{C-A-I} \exp\left(\frac{-\mathbb{X}}{u(1-u)}\right) \notag\\ 
&\quad \times (B)_{m} (B')_{n}\frac{{(zu)}^{m}{(wu)}^{n}}{m! n!}du.\label{3.5} 
\end{align}
Using the process discussed in \cite{ds1} we can show that the sequence of matrix functions in \eqref{3.5} is integrable and by dominated convergence theorem \cite{gf}, the summation and the integral can be interchanged in \eqref{3.5}. Now applying the matrix identity,
 \begin{align}
(1-z)^{-A}=\sum_{n=0}^{\infty}(A)_{n}\frac{z^{n}}{n!},\label{s11}
\end{align}
we can rewrite the Equation \eqref{3.5} as follows
\begin{align}
F_{1}(A, B, B'; C; z, w; \mathbb{X}) & = \Gamma\left(\begin{array}{c}C\\ A, C-A\end{array}\right) \int_{0}^{1} u^{A-I} (1-u)^{C-A-I} \exp\Big(\frac{-\mathbb{X}}{u(1-u)}\Big) \notag\\
 &\quad \times (1-zu)^{-B} (1-wu)^{-B'} du.
\end{align}
This completes the proof of Theorem~\ref{t1}.
\end{proof}
\begin{theorem}\label{ps1}
Let $A$, $B$, $B'$, $C$, $C'$, $C-B$, $C'-B'$ and $\mathbb{X}$ be  positive stable matrices in $\mathbb{C}^{r \times r}$ such that $B$, $B'$, $C$, $C'$  and $\mathbb{X}$ commutes with each other. Then the EAMF  $F_{2}(A, B, B'; C, C'; z, w; \mathbb{X})$  defined in \eqref{2eq2} has following integral representation: 
\begin{align}
&F_{2}(A, B, B'; C, C'; z, w; \mathbb{X})\notag\\&
=  \int_{0}^{1}\int_{0}^{1} (1-zu-wv)^{-A} u^{B-I} (1-u)^{C-B-I} v^{B'-I} (1-v)^{C'-B'-I} \notag\\
& \quad \times \exp \left( \frac{-\mathbb{X}}{u(1-u)}-\frac{\mathbb{X}}{v(1-v)}\right) du \, dv \,  \Gamma\left(\begin{array}{c}C, C'\\ B, B',  C-B, C'-B'\end{array}\right).\label{s33}
\end{align}
\end{theorem}
\begin{proof}
Using extended beta matrix function  and the EAMF $F_{2}(A, B, B'; C, C'; z, w; \mathbb{X})$  defined in \eqref{xb1} and \eqref{2eq2} respectively, we have
\begin{align}
&F_{2}(A, B, B'; C, C'; z, w; \mathbb{X})\notag\\
& = \sum_{m, n\geq 0}^{} \int_{0}^{1}\int_{0}^{1} (A)_{m+n} \frac{(zu)^{m} (wv)^{n}}{m! n!} u^{B-I} (1-u)^{C-B-I} v^{B'-I} (1-v)^{C'-B'-I}\notag\\
&\quad \times \exp \left(\frac{-\mathbb{X}}{u(1-u)}-\frac{\mathbb{X}}{v(1-v)}\right) \, du \, dv \, \Gamma\left(\begin{array}{c}C, C'\\ B, B',  C-B, C'-B'\end{array}\right).\label{3.9}
\end{align}
Summation and integral in \eqref{3.9} can be interchanged by using the dominated convergence theorem. Taking into account the summation formula \cite{SM}
\begin{align}
\sum_{N\geq 0}^{}f(N) \frac{(z+w)^{N}}{N!}=\sum_{m, n\geq 0}^{}f(m+n)\frac{z^{m}}{m!}\frac{w^{n}}{n!},
\end{align}
we get
\begin{align}
&F_{2}(A, B, B'; C, C'; z, w; \mathbb{X})\notag\\&
= \int_{0}^{1}\int_{0}^{1} \sum_{N\geq 0}^{}(A)_{N} \frac{(zu+wv)^{N}}{N!}  u^{B-I} (1-u)^{C-B-I} v^{B'-I} (1-v)^{C'-B'-I}\notag\\
& \quad \times \exp \Big(\frac{-\mathbb{X}}{u(1-u)}-\frac{\mathbb{X}}{v(1-v)}\Big)dudv \ \Gamma\left(\begin{array}{c}C, C'\\ B, B',  C-B, C'-B'\end{array}\right).\label{s22}
\end{align}
From \eqref{s11} and \eqref{s22}, we get \eqref{s33}.
\end{proof}
\begin{theorem}
Suppose $A$, $B$, $B'$, $B''$, $C$, $C-A$ and $\mathbb{X}$ be positive stable matrices in $\mathbb{C}^{r \times r}$ such that $A$, $C$, $\mathbb{X}$ commutes with each other and $CB = BC$, $CB' = B'C$, $CB'' = B'' C$. Then the ELMF $F^{(3)}_{D, \mathbb{X}}(A, B, B', B''; C; z, w, v)$  defined in \eqref{2eq3}  have the following integral representation: 
\begin{align}
F^{(3)}_{D, \mathbb{X}}(A, B, B', B''; C; z, w, v)& 
=  \Gamma\left(\begin{array}{c}C\\ A, C-A\end{array}\right) \int_{0}^{1} u^{A-I} (1-u)^{C-A-I} \exp\Big(\frac{-\mathbb{X}}{u(1-u)}\Big) \notag\\
& \quad \times (1-zu)^{-B} (1-wu)^{-B'} (1-vu)^{-B''} du.\label{3.12}
\end{align}
\end{theorem}
\begin{proof}
Equations \eqref{xb1} and \eqref{2eq3} together yield
\begin{align}
& F^{(3)}_{D, \mathbb{X}}(A, B, B', B''; C; z, w,v)\nonumber\\
& = \Gamma\left(\begin{array}{c} C \\ A, C-A\end{array}\right) \sum_{m, n, p\geq 0} \int_{0}^{1} u^{A-I} (1-u)^{C-A-I} \exp\Big(\frac{-\mathbb{X}}{u(1-u)}\Big) (B)_{m} (B')_{n} \nonumber\\
& \quad \times (B'')_{p} \frac{(uz)^{m} (uw)^{n} (uv)^{p}}{m! n! p!}.
\end{align}
Now, using the matrix relation \eqref{s11} and proceeding in the similar as in Theorem~\ref{t1}, we get the required result \eqref{3.12}.
\end{proof}
\section{Fractional calculus of extended hypergeometric matrix function}
The extended Riemann-Liouville fractional derivative of order $\mu$ is given by \cite{SM}
\begin{align}
&D_{z}^{\mu, p}\{f(z)\}\nonumber\\
 & =\frac{1}{\Gamma{(-\mu)}}\int_{0}^{z}f(t) (z-t)^{-\mu-1}\exp\left(\frac{-pz^{2}}{t(z-t)}\right)dt,\quad \Re(\mu)<0, \Re(p)>0 \label{r1}
\end{align}
and for $m-1<\Re(\mu)<m$ (m= 1, 2,\dots)
\begin{align}
D_{z}^{\mu, p}\{f(z)\}=&\frac{d^{m}}{dz^{m}}D_{z}^{\mu-m, p}\{f(z)\}\notag\\
=&\frac{d^{m}}{dz^{m}}\left(\frac{1}{\Gamma{(-\mu+m)}}\int_{0}^{z}f(t)(z-t)^{-\mu+m-1}\exp\left(\frac{-pz^{2}}{t(z-t)}\right)dt\right)\label{r2}
\end{align}
where the path of integration is a line from $0$ to $z$ in the complex t-plane. For the case $p=0$, we obtain the classical
Riemann-Liouville fractional derivative operator.

\begin{definition} 
Let  $\mathbb{X}$ be a positive stable matrix in $\mathbb{C}^{r\times r}$  and $\mu\in\mathbb{C}$ such that $\Re(\mu)<0$.  Then, the extended Riemann-Liouville fractional derivative of order $\mu$ is defined as follows
\begin{align}
D_{z}^{\mu, \mathbb{X}}\{f(z)\}=\frac{1}{\Gamma{(-\mu)}}\int_{0}^{z}f(t) (z-t)^{-\mu-1}\exp\left(\frac{-\mathbb{X}z^{2}}{t(z-t)}\right)dt \label{r11}
\end{align}
and for $m-1<\Re(\mu)<m$ (m= 1, 2,\dots)
\begin{align}
D_{z}^{\mu, \mathbb{X}}\{f(z)\}=&\frac{d^{m}}{dz^{m}}D_{z}^{\mu-m, \mathbb{X}}\{f(t)\}\notag\\
=&\frac{d^{m}}{dz^{m}}\left[\frac{1}{\Gamma{(-\mu+m)}}\int_{0}^{z}f(t) (z-t)^{-\mu+m-1} \exp \left(\frac{-\mathbb{X}z^{2}}{t(z-t)}\right)dt\right],\label{r21}
\end{align}
where the path of integration is a line from $0$ to $z$ in the complex t-plane. For the case $\mathbb{X}=O$ we obtain the classical
Riemann-Liouville fractional derivative operator.
\end{definition}
We start our investigation by calculating the extended fractional derivatives of some elementary functions.

\begin{theorem}\label{th1}
Let $A$ and $\mathbb{X}$ be  positive stable matrices in $\mathbb{C}^{r\times r}$  and $\mu\in\mathbb{C}$ such that $\Re(\mu)<0$.  Then, 
\begin{align}
D_{z}^{\mu, \mathbb{X}}\{z^{A}\}=z^{A-\mu I}\times \frac{\mathfrak{B}(A+I, -\mu I; \mathbb{X})}{\Gamma(-\mu)} 
\end{align}
\end{theorem}
\begin{proof}
From  \eqref{xb1} and \eqref{r11}, we obtain
\begin{align}D_{z}^{\mu, \mathbb{X}}\{z^{A}\}=&\frac{1}{\Gamma{(-\mu)}}\int_{0}^{z}t^{A} (z-t)^{-\mu-1}\exp\left(\frac{-\mathbb{X}z^{2}}{t(z-t)}\right)\,dt.\label{4.6}
\end{align}
Putting  $\frac{t}{z}=u$ and  $dt= z du$ in the equation \eqref{4.6}, we get
\begin{align}
D_{z}^{\mu, \mathbb{X}}\{z^{A}\} = &\frac{z^{A-\mu I}}{\Gamma(-\mu)}\int_{0}^{1}u^{A} (1-u)^{(-\mu-1)I}\exp\left(\frac{-\mathbb{X}}{u(1-u)}\right) du\notag\\
= & z^{A-\mu I} \ \frac{\mathfrak{B}(A+I, -\mu I; \mathbb{X})}{\Gamma(-\mu)}.\notag
\end{align}
This completes the proof.
\end{proof}
\begin{theorem}\label{q1}
Let $A$, $B$ and $\mathbb{X}$ be positive stable matrices in $\mathbb{C}^{r\times r}$  and $\lambda, \mu \in\mathbb{C}$ such that $\Re(\mu) > \Re(\lambda) > 0$.  Then, 
\begin{align}
&D_{z}^{\lambda-\mu, \mathbb{X}}\{z^{A-I}(1-z)^{-B}\}\notag\\
&= \Gamma(A) \Gamma^{-1}(A+(\mu-\lambda)I))z^{A+(\mu-\lambda-1)I} F^{(\mathbb{X})}(B, A; A+(\mu-\lambda)I; z).\label{ar13}
\end{align}
\end{theorem}
\begin{proof}
Proceeding similarly as in Theorem~\ref{th1}, we get
\begin{align}
D_{z}^{\lambda-\mu, \mathbb{X}}\{z^{A-I}(1-z)^{-B}\}
& = \frac{z^{A+(\mu-\lambda-1)I}}{\Gamma(\mu-\lambda)} \int_{0}^{1}u^{A-I}(1-uz)^{-B}(1-u)^{\mu-\lambda-1}\nonumber\\
& \quad \times \exp\left(\frac{-\mathbb{X}}{u(1-u)}\right) du.
\end{align}
Using \eqref{ieg1}, we can write
\begin{align}
D_{z}^{\lambda-\mu, \mathbb{X}}\{z^{A-I}(1-z)^{-B}\}=&\frac{z^{A+(\mu-\lambda-1)I}}{\Gamma(\mu-\lambda)}\mathfrak{B}(A, (\mu-\lambda)I) F^{(\mathbb{X})}(B, A; A+(\mu-\lambda)I; z)\notag\\
=&z^{A+(\mu-\lambda-1)I} \Gamma(A)\Gamma^{-1}(A+(\mu-\lambda)I)\,F^{(\mathbb{X})}(B, A; A+(\mu-\lambda)I; z).\notag
\end{align}Hence the proof is completed.
\end{proof}
The transform of several other matrix functions can be obtained easily by using the extended Riemann-Liouville fractional derivative operators defined in \eqref{r11} and \eqref{r21}. Since the proofs are similar to Theorems~\ref{th1} and \ref{q1}, we produce the results without proofs.
\begin{theorem}\label{th2}
Let $A$, $B$, $B'$and  $C$ be matrices  in $\mathbb{C}^{r\times r}$ such that  $AB=BA$,  $AB'=B'A$, $AC=CA$; $A$, $C$ and $C-A$  are positive stable and $\lambda, \mu\in\mathbb{C}$ such that   $\Re(\mu)>\Re(\lambda)>0$; $|az|<1, |bz|<1$.  Then
\begin{align}
&D_{z}^{\lambda-\mu, \mathbb{X}}\{z^{A-I}(1-az)^{-B}(1-bz)^{-B'}\}\notag\\&= \Gamma(A) \Gamma^{-1}(A+(\mu-\lambda)I))z^{A+(\mu-\lambda-1)I} F_{1}(A, B, B'; A+(\mu-\lambda)I; az, bz; \mathbb{X}).\label{er13}
\end{align}
\end{theorem}
\begin{theorem}
Let $A$, $B$, $B'$, $B''$  and $C$ be matrices in  $\mathbb{C}^{r\times r}$ such that $AC=CA$,  $A$ and $C$ commute with  all remain matrices; $A$, $C$ and $C-A$  are positive stable  and $\lambda, \mu\in\mathbb{C}$ such that $\Re(\mu)>\Re(\lambda)>0$; $|az|<1, |bz|<1$, $|cz|<1$.  Then 
\begin{align}
&D_{z}^{\lambda-\mu, \mathbb{X}}\{z^{A-I}(1-az)^{-B}(1-bz)^{-B'}(1-cz)^{-B''}\}\notag\\&= \Gamma(A) \Gamma^{-1}(A+(\mu-\lambda)I))z^{A+(\mu-\lambda-1)I} F^{(3)}_{D, \mathbb{X}}(A, B, B', B''; A+(\mu-\lambda)I; az, bz, cz).\label{r13}
\end{align}
\end{theorem}
\begin{theorem}Let  $A$, $B$, $B'$ and $C$ be   matrices in  $\mathbb{C}^{r\times r}$ such that $B$, $B'$, $C$ commute with each other and  $B$, $B'$, $C$,  $C-B$ and $C'-B'$ are positive stable. Also let $\lambda, \mu\in\mathbb{C}$ such that   $\Re(\mu)>\Re(\lambda)>0$. Then, for $|\frac{x}{1-z}|<1$, we have 
\begin{align}
&D_{z}^{\lambda-\mu, \mathbb{X}}\{z^{B'-I}(1-z)^{-A}F^{(\mathbb{X})}(A, B; C; \frac{x}{1-z})\}\notag\\&=\frac{z^{B'+(\mu-\lambda-1)I}}{\Gamma(\mu-\lambda)} F_{2}(A, B, B'; C, \mu I; x, z; \mathbb{X})\times \Gamma\left(\begin{array}{c}C\\ B,  C-B\end{array}\right).\label{4.11}
\end{align}
\end{theorem}
\section{Generating functions}
\begin{theorem}Let $A$, $B$ and $\mathbb{X}$ be positive stable matrices in $\mathbb{C}^{r\times r}$  and $\lambda, \mu\in\mathbb{C}$ such that $\Re(\mu)>\Re(\lambda)>0$. Then, for $|z|<|1-t|$, we have
\begin{align}
\sum_{n=0}^{\infty}\frac{(A)_{n}}{n!}F^{(\mathbb{X})}(A+nI, B; B+(\mu-\lambda) I; z)\, t^{n}= (1-t)^{-A} F^{(\mathbb{X})}{(A, B;  B+(\mu-\lambda) I; \frac{z}{1-t})}.\label{a1}
\end{align}
\end{theorem}
\begin{proof}
Consider the matrix identity
\begin{align}
[(1-z)-t]^{-A}= (1-t)^{-A}\Big[1-\frac{z}{1-t}\Big]^{-A},\notag
\end{align}
which can be written as
\begin{align}
\sum_{n=0}^{\infty}\frac{(A)_{n}}{n!} (1-z)^{-A}\left(\frac{t}{1-z}\right)^{n}= (1-t)^{-A}\Big[1-\frac{z}{1-t}\Big]^{-A}.\label{e5.2}
\end{align}
Now, multiplying by $z^{B-I}$ in Equation \eqref{e5.2} and applying the extended fractional derivative operator $D_{z}^{\lambda-\mu, \mathbb{X}}$, we can write
\begin{align}
&\sum_{n=0}^{\infty}\frac{(A)_{n}}{n!}D_{z}^{\lambda-\mu, \mathbb{X}}\{z^{B-I} (1-z)^{-A-nI}\}t^{n}= (1-t)^{-A}D_{z}^{\lambda-\mu, \mathbb{X}}\{z^{B-I}\Big[1-\frac{z}{1-t}\Big]^{-A}\}\notag.
\end{align}
Using Theorem {\ref{q1}}, we get the desired result.
\end{proof}
\begin{theorem}\label{t5.2}
Let $A$, $B$, $B'$ and $\mathbb{X}$ be positive stable matrices in $\mathbb{C}^{r\times r}$  and $\lambda, \mu\in\mathbb{C}$ such that $\Re(\mu)>\Re(\lambda)>0$. Then, for $|t|<\frac{1}{1+|z|}$, we have the following generating relation
\begin{align}
&\sum_{n=0}^{\infty}\frac{(B')_{n}}{n!}F^{(\mathbb{X})}(B-nI, A; A+(\mu-\lambda) I; z)\, t^{n}\notag\\&= (1-t)^{-B'} F_1{(A, B, B';  A+(\mu-\lambda) I; z, \frac{-zt}{1-t}; \mathbb{X})}.\notag
\end{align}
\end{theorem}
\begin{proof}
Using \eqref{e5.2}, we have
\begin{align}
\sum_{n=0}^{\infty}\frac{(B')_{n}}{n!}(1-z)^{n} t^{n}= (1-t)^{-B'}\Big[1-\frac{-zt}{1-t}\Big]^{-B'}.\label{5.3}
\end{align}
Now multiplying both sides of the Equation \eqref{5.3} by $z^{A-I}(1-z)^{-B}$ and applying the extended fractional derivative operator $D_{z}^{\lambda-\mu, \mathbb{X}}$, we get
\begin{align}
&\sum_{n=0}^{\infty}\frac{(B')_{n}}{n!}D_{z}^{\lambda-\mu, \mathbb{X}}\{z^{A-I}(1-z)^{-(B-nI)}\}t^{n}\notag\\&= (1-t)^{-B'}D_{z}^{\lambda-\mu, \mathbb{X}}\{z^{A-I}(1-z)^{-B}\Big[1-\frac{-zt}{1-t}\Big]^{-B'}\}\notag.
\end{align}
Using the Theorems {\ref{q1}} and {\ref{th2}}, we get the desired result. 
\end{proof}
\section{Conclusion}
In this paper, we studied the generalized version of some Appell matrix function and Lauricella matrix function. We obtained some generating relations in terms of matrix functions using the extended Riemann-Liouville fractional derivative of order $\mu$. The particular case of our results, i.e. if we take matrices from $\mathbb{C}^{1 \times 1}$,  coincides with the results obtained in \cite{ao}.

\end{document}